\documentclass[a4paper, oneside, reqno, 10pt]{amsart}
\usepackage{graphicx}
\usepackage{amsmath}
\usepackage{amsthm}
\usepackage{amsfonts}
\usepackage{geometry}
\usepackage[T1]{fontenc}
\usepackage{lmodern}
\usepackage{xcolor}
\usepackage{tikz}
\usepackage{tikz-cd}
\usetikzlibrary{arrows.meta, positioning, calc, shapes, shapes.geometric, patterns, decorations.pathreplacing, decorations.pathmorphing, matrix, fit, backgrounds, plotmarks}
\usepackage{mathtools}
\usepackage{bm}
\usepackage{mathrsfs}
\usepackage{dsfont}
\usepackage{bbm}
\usepackage{soul}
\usepackage{cancel}
\usepackage{comment}

\usepackage[backref,bookmarks=true,colorlinks=true,pdfstartview=FitV,linkcolor=blue, citecolor=red,urlcolor=green]{hyperref}  

\usepackage{setspace}

\newtheorem{theorem}{Theorem}[section]

\newtheorem{lemma}[theorem]{Lemma}

\newtheorem{claim}[theorem]{Claim}

\theoremstyle{definition}
\newtheorem{definition}[theorem]{Definition}

\theoremstyle{remark}

\newcommand{\Hit}{\mathrm{Hit}}
\newcommand{\bR}{\mathbb{R}}

\newcommand{\PSL}{\mathrm{PSL}}

\newcommand{\length}{\ell}
\newcommand{\moduli}{\mathcal{M}}
\newcommand{\Mod}{\mathrm{Mod}}

\newcommand{\pants}{\mathcal{P}}
\newcommand{\Hom}{\mathrm{Hom}}
\newcommand{\diag}{\mathrm{diag}}

\begin{document}
\title{The Thick Part of the $\mathrm{PSL}_n(\mathbb{R})$-Hitchin-Riemann Moduli Space has Infinite Volume}

\author[An]{Huiseong An}
\address{\hskip-\parindent Department of Mathematics\\Sungkyunkwan University}
\email{huiseongan6@gmail.com}

\author[Kim]{Heejae Kim}
\address{\hskip-\parindent Department of Medicine\\Sungkyunkwan University}
\email{heejaek31@gmail.com}

\author[Kim]{Mingook Kim}
\address{\hskip-\parindent Department of Mathematics\\Sungkyunkwan University}
\email{mathmgk5150@gmail.com
}

\author[Lee]{SEUNGJIN LEE}
\address{\hskip-\parindent Department of statistics and Data science\\Yonsei University}
\email{tmdwls5648@gmail.com}

\author[Jung]{Hongtaek Jung}
\address{\hskip-\parindent
School of Mathematics\\KIAS}
\email{htjung1905@gmail.com}

\begin{abstract}
    We prove that the thick part of the $\PSL_n(\bR)$-Hitchin--Riemann moduli space has infinite total Atiyah--Bott--Goldman volume for $n>2$. This result stands in contrast to Mumford's compactness criterion. To achieve this result, we employ Goldman flows and internal sequences to find an infinite series of subsets of identical volume, the images of which in the Hitchin--Riemann moduli space are all mutually disjoint and sit in the thick part. 
\end{abstract}

\maketitle

\section{Introduction}

    In higher Teichm\"uller theory, the $\PSL_n(\bR)$-Hitchin component $\Hit_n(\mathrm{S})$ is a central object of study, serving as a natural higher-rank generalization of the classical Teichm\"uller space. A fundamental question is whether the geometry of the $\PSL_n(\bR)$-Hitchin component parallels that of the classical Teichm\"uller space. For example, Hitchin \cite{Liegroups} shows that the Hitchin components have the same topological type as the Teichm\"uller space; Labourie \cite{Labourie} proves that Hitchin representations and Fuchsian representations share similar dynamical properties; Sun, Wienhard and Zhang \cite{SZ,SWZ} explore symplectic-geometric similarities between the $\PSL_n(\bR)$-Hitchin component and the Teichm\"uller space.
    
    Consider the action of the mapping class group $\Mod(\mathrm{S})$ on the Teichm\"uller space and the Hitchin components. It is a classical result that the mapping class group acts properly discontinuously on the Teichm\"uller space and its quotient space yields the moduli space of Riemann surfaces $\mathcal{M}(\mathrm{S})$. Similarly, it is known  \cite{Lab_cross}  that $\Mod(\mathrm{S})$ acts properly discontinuously on $\Hit_n(\mathrm{S})$. This defines a smooth orbifold  $\mathcal{M}_n(\mathrm{S})=\Hit_n(\mathrm{S})/\Mod(\mathrm{S})$ which we call the $\PSL_n(\bR)$-Hitchin--Riemann moduli space. 
    
    Regarding the topology of the Riemann moduli space, Mumford's compactness criterion states that the $\epsilon$-thick part of  $\mathcal{M}(\mathrm{S})$ is compact. On the other hand, one can naturally generalize the notion of the $\epsilon$-thick part of the moduli space to the Hitchin--Riemann moduli space; see Definition~\ref{thick part}. Given the analogies between the Teichm\"uller space and the Hitchin component, it is natural to ask whether the $\epsilon$-thick part of the $\PSL_n(\bR)$-Hitchin--Riemann moduli space is compact. Our main result, Theorem {\ref{thm:main}}, provides a stronger answer to this question. 
    
\begin{theorem}[Main Theorem]\label{thm:main}
    Let $\mathrm{S}$ be a closed orientable surface of genus $g>1$ and let $\epsilon$ be any positive real number. Then the total Atiyah--Bott--Goldman volume of the $\epsilon$-thick part of the $\PSL_n(\bR)$-Hitchin-Riemann moduli space is infinite provided $n>2$.
\end{theorem}

    Choi--Jung \cite{Choi2025volume} show that the same conclusion holds for $\epsilon$-thin parts of various higher Teichm\"uller spaces.

    Since the Atiyah--Bott--Goldman volume is locally finite, Theorem~\ref{thm:main} immediately implies that the $\epsilon$-thick part of the Hitchin--Riemann moduli space is noncompact. 
    
    Theorem~\ref{thm:main} is established by directly exploiting the parametrization of the Hitchin component. Specifically, our proof utilizes the ``internal sequences'' introduced in the work of \cite{Zhang2015}. By carefully analyzing these internal sequences, we construct families of open sets with equal positive volume that remain  within the thick part of the Hitchin component, yet pairwise disjoint in the moduli space thereby yielding infinite volume.

\subsection*{Materials and Methods}
The authors acknowledge the use of ChatGPT, and Gemini for writing assistance and literature searches.

\section{Preliminaries}
We recall the notion of Hitchin representations and the mapping class group. Also we define thick parts of Hitchin--Riemann moduli spaces.

Throughout this section, $\mathrm{S}$ will be a compact oriented connected surface of negative Euler characteristic unless otherwise stated.
 
\subsection{Hitchin--Riemann moduli spaces}
\begin{definition}
      A Fuchsian representation is any representation of the form $\tau_n \circ \rho_0$, where $\rho_0 : \pi_1(\mathrm{S}) \to \PSL_2(\bR)$ is a holonomy representation of a hyperbolic metric on $\mathrm{S}$, and $\tau_n :\PSL_2(\bR) \to \PSL_n(\bR)$ is the irreducible representation.
\end{definition}

\begin{definition}
A representation $\rho:\pi_1(\mathrm{S}) \to \PSL_n(\bR)$ is called a $\PSL_n(\bR)$-\emph{Hitchin representation} if there exists a continuous family $\{\rho_t\}_{0\le t \le 1}$ of representations such that 
\begin{enumerate}
    \item $\rho_0 = \rho$ and $\rho_1$ is a Fuchsian representation;
    \item For each oriented boundary component $\zeta$ of $\mathrm{S}$, $\rho_t(\zeta)$ is conjugate to a matrix of the form $\diag(d_1, \cdots, d_n)$ with $d_1>d_2>\cdots>d_n>0$.
\end{enumerate} 
\end{definition}

Denote by 
\[
\Hom_{\mathrm{H}}(\pi_1(\mathrm{S}), \PSL_n(\bR))
\]
the space of $\PSL_n(\bR)$-Hitchin representations. This space is connected when $n$ is odd and has two connected component when $n$ is even. Let $\Hit_n(\mathrm{S})$ be a connected component of
\[
\Hom_\mathrm{H}(\pi_1(\mathrm{S}),\PSL_n(\bR))/\PSL_n(\bR).
\]
We call $\Hit_n(\mathrm{S}$) the $\PSL_n(\bR)$-\emph{Hitchin component}.  

The \emph{mapping class group} $\Mod(\mathrm{S})$ is the group of isotopy classes of orientation-preserving homeomorphisms of $\mathrm{S}$, that is, $\Mod(\mathrm{S}):=\pi_0(\mathrm{Homeo^+}(\mathrm{S}))$.

We describe the action of $\Mod(\mathrm{S})$ on $\Hit_n(\mathrm{S})$. For this, choose a representative $f$ of $\phi \in \Mod(\mathrm{S})$. Let $f_\#$ be the automorphism of $\pi_1(\mathrm{S})$ induced by the homeomorphism $f$ of $\mathrm{S}$. Then the action of $\phi$ is defined by $\phi\cdot[\rho]=\phi_*([\rho])= [\rho \circ f_{\#}^{-1}]$ for $[\rho] \in \Hit_n(\mathrm{S})$. This action is known to be properly discontinuous and effective \cite[Theorem~1.0.2]{Lab_cross}.

The Hitchin--Riemann moduli space, denoted $\mathcal{M}_n(\mathrm{S})$, is defined as the quotient space of the Hitchin component under the action of the mapping class group; that is,
\[
\mathcal{M}_n(\mathrm{S})=\Hit_n(\mathrm{S})/\Mod(\mathrm{S}).
\]

Since $\Mod(\mathrm{S})$ is a countable group, we can apply the following lemma.
\begin{lemma}\label{lem:2.3}
    Let $G$ be a countable group acting properly discontinuously and effectively on a connected manifold $M$. Let $U$ be any nonempty open set of $M$. Then there is a nonempty open set $U' \subset U$ such that
    \begin{align*}
        \{g \in G \mid g(U') \cap U' \neq \emptyset \} = \{1\}.
    \end{align*}
\end{lemma}

\subsection{Thick parts of the Hitchin components}
For $n=2$, $\Hit_2(\mathrm{S})$ is nothing but the Teichm\"uller space. In this setting, we can define the hyperbolic length of isotopy classes of essential simple closed curves in $\mathrm{S}$ by $
    \ell_\rho^{\mathrm{hyp}}(\gamma) = \log\left|\frac{\omega_1(\rho(\gamma))}{\omega_2(\rho(\gamma))}\right|$, where $\gamma \in \pi_1(\mathrm{S})\setminus\{1\}$, and $\omega_1(\rho(\gamma))$, $ \omega_2(\rho(\gamma))$ are eigenvalues of $\rho(\gamma)$ with $\omega_1(\rho(\gamma)) > \omega_2(\rho(\gamma))$. Similarly, we can define the length function for $\PSL_n(\bR)$-Hitchin representations.

\begin{definition}
    The \emph{length function} $\ell_\rho:\pi_1(\mathrm{S}) \to \bR$ with respect to $\left[\rho\right]\in  \Hit_n(\mathrm{S})$ is defined by 
    \[
    \ell_\rho(\gamma) = \log\left|\frac{\omega_{1}(\rho(\gamma))}{\omega_{n}(\rho(\gamma))}\right|
    \] 
    where $\omega_{1}(\rho(\gamma))$, $\omega_{n}(\rho(\gamma))$ are the largest and the smallest eigenvalues of $\rho(\gamma)$, respectively.
    
    Slightly abusing the notation, for any unoriented essential simple closed curve $\alpha$ in $\mathrm{S}$, we denote by $\ell_{\rho}(\alpha)$ the value $\ell_{\rho}(\alpha')$, where $\alpha'$ is an element of $\pi_1(\mathrm{S})$ that is freely isotopic to $\alpha$. This is well defined because the length function is invariant under conjugation and the orientation reversal. 
\end{definition}

This definition is a natural generalization of the hyperbolic length in the sense that for any Fuchsian representation $\rho=\tau \circ \rho_0$, we have $\ell_{\rho}(\gamma)=(n-1)\cdot  \ell_{\rho_0}^{\mathrm{hyp}}(\gamma)$.

Recall that, for a positive $\epsilon>0$, the $\epsilon$-thick part of the Teichm\"uller space $\mathcal{T}(\mathrm{S})$ is defined by
\[
\mathcal{T}^{\epsilon}(\mathrm{S}) := \{[\rho]\in \mathcal{T} (\mathrm{S}) \mid \ell_{\rho}^{\mathrm{hyp}}(\gamma)\ge \epsilon \ \text{for all essential closed curves} \ \gamma \ \text{in} \ \mathrm{S} \}.
\] 
Since $\mathcal{T}^\epsilon(\mathrm{S})$ is invariant under the mapping class group action, we can define the $\epsilon$-thick part of the Riemann moduli space to be $\mathcal{M}^\epsilon(\mathrm{S})= \mathcal{T}^\epsilon(\mathrm{S})/\Mod(\mathrm{S})$.

By replacing $\ell^{\mathrm{hyp}}_\rho$ with $\ell_\rho$, one can naturally extend these definitions to $\Hit_n(\mathrm{S})$. More precisely:
\begin{definition}\label{thick part}
    For a given $\epsilon >0$, the $\epsilon$-\emph{thick part} of the Hitchin component is defined by
    \[
        \Hit^{\epsilon}_{n}(\mathrm{S})  = \{\left[\rho\right] \in \Hit_n(\mathrm{S}) \mid  \ell_\rho(\gamma)\ge \epsilon \ \text{for all essential closed curves} \ \gamma \ \text{in} \ \mathrm{S}\}.
    \]
We also define the $\epsilon$-thick part of the Hitchin--Riemann moduli space to be 
    \[
\mathcal{M}^{\epsilon}_{n}(\mathrm{S})=\Hit^{\epsilon}_{n}(\mathrm{S})/ \Mod(\mathrm{S}).
    \]
\end{definition}

\subsection{Parametrizations}\label{parametrization}
In this paper, we construct an infinite volume open set in $\moduli^{\epsilon}_n(\mathrm{S})$ based on a parametrization of $\Hit_n(\mathrm{S})$ given by  \cite{Bonahon2014} and \cite{Zhang2015}. 

We only outline the parametrization. For a given \emph{closed} oriented surface $\mathrm{S}$, we first take an oriented pair-of-pants decomposition $\pants=\{\alpha_1, \cdots, \alpha_{3g-3}\}$ of $\mathrm{S}$ and extend it to a maximal geodesic lamination. For each $\alpha_i$, we associate a real analytic map $\pi_{b_i}:\Hit_n(\mathrm{S}) \to \mathfrak{a}^+$ defined by 
\[
\pi_{b_i}([\rho]) = \diag(\log|\omega_1(\rho(\alpha_i))|,\cdots, \log |\omega_n(\rho(\alpha_i))|)
\]
where
\[
\mathfrak{a}^+=\left\{\diag(d_1, \cdots, d_n)\mid \sum d_i = 0,\,d_1>d_2>\cdots >d_n\right\}.
\]
Also, for each $\alpha_i$, we assign $n-1$ gluing parameters, giving rise to a real analytic map $\pi_{s_i}:\Hit_n(\mathrm{S})\to \mathcal{G}_i=\bR^{n-1}$. 

For each pair-of-pants $P_i$, $i=1,2,\cdots, 2g-2$, one can find $n^2-1$ triangle-shear invariants satisfying certain linear equations as in \cite{Bonahon2014}. Precise definition of triangle-shear invariants is not important for our purposes. In \cite[Proposition 3.5]{Zhang2015}, the author systematically chooses $(n-1)(n-2)$ free parameters among these triangle-shear invariants obtaining an analytic map $\pi_{P_i}:\Hit_n(\mathrm{S}) \to \mathcal{I}_i=\bR^{(n-1)(n-2)}$, for $i=1,2,\cdots, 2g-2$. 

By combining these maps $\pi_{b_i}$, $\pi_{s_i}$, and $\pi_{P_i}$, we have  a real analytic diffeomorphism 
    \[\Xi:\Hit_n(\mathrm{S}) \to (\mathfrak{a}^+)^{(3g-3)} \times \mathcal{I}\times  \mathcal{G},
    \]
where $\mathcal{I}=\mathcal{I}_1\times\cdots\times\mathcal{I}_{2g-2}$ and $\mathcal{G}=\mathcal{G}_1\times\cdots\times\mathcal{G}_{3g-3}$.

\subsection{Restriction of the Hitchin component}\label{restriction} 
Let $\alpha$ be an oriented, essential, simple closed, non-separating curve on $\mathrm{S}$. Let $\nu(\alpha)$ be a tubular neighborhood of $\alpha$, and set $\mathrm{S}_1 = \mathrm{S} \setminus \nu( \alpha )$. Let $\alpha^+$ and $\alpha^-$ be the two distinct boundary components of $\mathrm{S}_1$, oriented such that the quotient map $q : \mathrm{S}_1 \to \mathrm{S}$ restricts to orientation-preserving homeomorphisms from $\alpha^+$ and $\alpha^-$ to $\alpha$.
Then the fundamental group can be presented as an HNN extension
\[
\pi_1(\mathrm{S}) = \langle \pi_1(\mathrm{S}_1), v \mid v[\alpha^+]v^{-1} = [\alpha^-] \rangle,
\]
where $[\alpha^+]$ and $[\alpha^{-}]$ are elements of $\pi_1(\mathrm{S}_1)$ that are freely homotopic to $\alpha^+$ and $\alpha^-$ respectively, and where $v$ may be represented by an essential oriented simple closed curve in $\mathrm{S}$ that intersects $\alpha$ exactly once positively.

The inclusion map $i : \pi_1(\mathrm{S}_1) \to \pi_1(\mathrm{S})$ induces the map 
\[\tilde{\mathcal{R}} : \Hom_{\mathrm{H}}(\pi_1(\mathrm{S}), \PSL_n({\bR})) \to \Hom_{\mathrm{H}}(\pi_1(\mathrm{S}_1), \PSL_n({\bR}))\] 
defined by $\tilde{\mathcal{R}}(\rho)  = \rho \circ i$. Let $\tilde{\Delta}$ be the image of $\tilde{\mathcal{R}}$ in $\Hom_{\mathrm{H}}(\pi_1(\mathrm{S}_1),\PSL_n({\bR}))$. This again yields the restriction map
\[\mathcal{R} : \Hit_n(\mathrm{S}) \to \Delta = \tilde{\Delta}/ \PSL_n(\bR).
\]
Indeed, the space $\Delta$ is given by
\[
\Delta=\{[\rho]\in \Hit_n(\mathrm{S}_1)\mid [\rho(\alpha^+)] = [\rho(\alpha^-)]\}.
\]
Assuming that $\mathrm{S}$ is closed, the parametrization in Section~\ref{parametrization} implies that $\mathcal{R}$ is an open map and that $\Delta$ is a connected smooth manifold. 

Let us define 
\[
\Mod(\mathrm{S}_1, \partial{\mathrm{S}_1}):=\pi_0(\{\phi \in \mathrm{Homeo}^+(\mathrm{S}_1) \mid \phi(\partial \mathrm{S}_1)=\partial \mathrm{S}_1\}).
\]
We know that $\Mod(\mathrm{S}_1,\partial \mathrm{S}_1)$ acts on $\Delta$ properly discontinuously. See \cite[Section~2.5]{Choi2025volume} for more details. 

\subsection{Goldman flows}
Let $\alpha$ and $\gamma$ be essential, oriented simple closed curves on $\mathrm{S}$. Geometrically, a Goldman flow is obtained by cutting $\mathrm{S}$ along
$\alpha$ and regluing the resulting boundary components after twisting by an
element of the centralizer of the holonomy of $\alpha$.

Suppose that $\alpha$ is non-separating. As in Section~\ref{restriction},
the fundamental group can be presented as an HNN extension
\[
\pi_1(\mathrm{S})
=
\left\langle
\pi_1(\mathrm{S}_1),v
\;\middle|\;
v[\alpha^+]v^{-1}=[\alpha^-]
\right\rangle.
\]

Let $[\rho]\in\Hit_n(\mathrm{S})$. Choose a representation $\rho$ in the conjugacy class $[\rho]$ so that $\rho(\alpha)$ is diagonal. Choose
\[
Q\in
\mathfrak z_{\mathfrak{sl}_n(\mathbb R)}(\rho(\alpha))
=
\left\{
X\in\mathfrak{sl}_n(\mathbb R)
\;\middle|\;
\rho(\alpha)X\rho(\alpha)^{-1}=X
\right\}.
\]
Since $\exp(Q)$ commutes with $\rho(\alpha)$, we can find a representation $\rho_Q:\pi_1(\mathrm{S})\to\PSL_n(\mathbb R)$ satisfying
\[
\rho_Q(\gamma)
=
\begin{cases}
\rho(\gamma) & \gamma\in\pi_1(\mathrm{S}_1) \\
\rho(v)\exp(Q) & \gamma=v
\end{cases}.
\]
The map
\[
\Phi^{Q,\alpha}:\Hit_n(\mathrm{S})\longrightarrow\Hit_n(\mathrm{S})
\]
defined by $\Phi^{Q,\alpha}([\rho]) = [\rho_Q]$ preserves the defining relations of $\pi_1(\mathrm{S})$ so $\Phi^{Q,\alpha}([\rho])$ is again a representation. The flow
\[
\Phi_t ^{Q,\alpha} = \Phi^{tQ, \alpha}
\]
is called a \emph{Goldman flow} along $\alpha$, originally introduced in \cite{Goldman86hamiltonian}.

\begin{definition}
Fix a matrix
\[
Q=
\begin{cases}
\diag(1,-2,1) & n=3 \\
\diag(1,0, \cdots,-1) & n\geq 4
\end{cases}
\]
in $\mathfrak{z}_{\mathfrak{sl}_n(\bR)}(\rho(\alpha))$ and let $t \in \bR$. The \emph{bulging flow} along $\alpha$ is the one-parameter family of maps
\[
\Phi_t : \Hit_n(\mathrm{S}) \to \Hit_n(\mathrm{S})
\]
defined by
\[
\Phi_t ([\rho]) = \Phi^{tQ,\alpha}([\rho]).
\]
\end{definition}
Goldman \cite{Goldman86hamiltonian} showed that  $\Phi_t$ preserves the Atiyah--Bott--Goldman volume on $\Hit_n(\mathrm{S})$.

\subsection{Internal sequences}

In this section, we introduce one of the key ingredients of the proof. For  technical reasons we will assume that $\mathrm{S}$ is closed for this subsection.

Let $\pants$ denote an oriented pair-of-pants decomposition $\pants=\{\alpha_1, \cdots, \alpha_{3g-3}\}$ of $\mathrm{S}$.

\begin{definition}[\cite{Zhang2015}]\label{internal} A sequence $\{\left[\rho_i\right]\} \in \Hit_n(\mathrm{S})$ is called an \emph{internal sequence} if
\begin{enumerate}
    \item There exist two positive constants $D_1$, $D_2$ such that for all $\gamma\in \Gamma_\pants$, 
    \begin{align*}
        D_1<\lambda_k(\rho_i(\gamma))-\lambda_{k+1}(\rho_i(\gamma))<D_2
    \end{align*}
for all $1\leq k \leq n-1$ and all $i\geq 0$. For here, $\Gamma_\pants$ is the set of elements in the  fundamental group represented by the members of $\pants$, and $\lambda_k=\log |\omega_k|$ is the log of the $k$-th largest eigenvalue of the matrix.
    \item For each $1\leq j \leq 2g-2$ and for any compact subset $H \subset \mathcal{I}_j$, there exists some integer $N$ such that $\pi_{P_j}(\rho_i) \notin H$ for $i>N$.
\end{enumerate}
\end{definition}

    Now, we briefly introduce two quantities $r(\psi(\gamma))$ and $K(\rho)$, which are discussed in \cite{Zhang2015}. The quantity $r(\psi(\gamma))$ captures the combinatorial crossing behavior of the closed geodesic associated with $\gamma$, relative to the fixed pants decomposition $\pants$. Roughly speaking, it measures the frequency with which the geodesic crosses between collar neighborhoods of pants curves. The quantity $K(\rho)$ measures the minimal length contribution arising from crossing between adjacent pairs of pants.
     
\begin{theorem}[Theorem 4.19 of \cite{Zhang2015}]\label{thm:zhang1}
    Pick $\left[\rho\right] \in \Hit_n(\mathrm{S})$, $\gamma\in \pi_1(\mathrm{S})$ such that $\gamma \notin \langle A \rangle$ for every $A \in \Gamma_\pants$.
    Then,
    \[
        \ell_{\rho}(\gamma) \geq r(\psi(\gamma))\frac{K([\rho])}{11}.
    \]
    \end{theorem}

The following result is essential for our proof.

\begin{theorem}[Theorem 5.1 of \cite{Zhang2015}]\label{thm:zhang2}
    Let $\{\left[\rho_i\right]\}_{i=1}^\infty$ be an internal sequence in $\Hit_n(\mathrm{S})$. Then \[
    \lim_{i\to\infty}K(\left[\rho_i\right]) = \infty
    \]
    Therefore, \[
    \ell_{\rho_i}(\gamma) \to \infty \quad \text{as} \quad i\to\infty
    \] for any $\gamma \notin \langle A \rangle$ for every $A \in \Gamma_\pants$.
\end{theorem} 
Thanks to Theorem~\ref{thm:zhang2}, we can make the length of any closed curve not isotopic to a curve in the pants decomposition arbitrarily large without contracting the lengths of elements in $\Gamma_\pants$. 

\section{Proof of the Main Theorem}
In this section we prove the main theorem. Let $\mathrm{S}$ be a closed oriented surface of genus $g>1$, and let $\Gamma=\pi_1(\mathrm{S})$ be the fundamental group of $\mathrm{S}$. Let us fix an oriented pair-of-pants decomposition $\pants=\{\alpha_1, \cdots, \alpha_{3g-3}\}$ of $\mathrm{S}$ that contains a non-separating, oriented essential simple closed curve $\alpha$.

Our proof closely follows the argument of \cite{Choi2025volume}. While the collar lemma, \cite[Theorem~4.2]{Choi2025volume} can be applied to the thin part, it is not applicable to the thick part. Therefore, we construct a nice open set $N$ of \cite[Proposition~4.3]{Choi2025volume} by using an internal sequence~\ref{internal}.

\begin{lemma}\label{prop:main}
    Let $\mathrm{S}$ be a closed oriented surface of genus $g>1$. Let 
    \[
    \Xi:\Hit_n(\mathrm{S}) \to (\mathfrak{a}^+)^{(3g-3)} \times \mathcal{I} \times \mathcal{G}
    \]
    be the parametrization obtained by $\pants$ and a maximal geodesic lamination as in Section~\ref{parametrization}. Then $\mathcal{I}$ is invariant pointwise under the bulging flow $\Phi_t:\Hit_n(\mathrm{S}) \to \Hit_n(\mathrm{S})$ along any non-separating curve $\alpha$ in $\pants$.
\end{lemma}

\begin{proof}
     Let $[\rho]\in\Hit_n(\mathrm{S})$. A Goldman flow $\Phi_t$ along $\alpha$ maps $[\rho]$ to $\Phi_t([\rho])\in\Hit_n(\mathrm{S})$. Then there are two Frenet curves $\xi_\rho$ and $\xi_{\Phi_t(\rho)}$ corresponding to $\rho$ and $\Phi_t(\rho)$, respectively. Let $\mathrm{S}_1=\mathrm{S} \setminus \nu(\alpha)$ where $\nu(\alpha)$ is a tubular neighborhood of $\alpha$. Regard $\pi_1(\mathrm{S}_1)$ as a subgroup of $\pi_1(\mathrm{S})$. Consider the restricted Frenet curves $\xi_\rho |_{ \partial \pi_1(\mathrm{S}_1)}$ and $\xi_{\Phi_t(\rho)}|_{\partial \pi_1(\mathrm{S}_1)}$. Since the Goldman flow does not change the representation $\rho|\pi_1(\mathrm{S}_1)$, we see that $\xi_\rho |_{ \partial \pi_1(\mathrm{S}_1)}=\xi_{\Phi_t(\rho)}|_{\partial \pi_1(\mathrm{S}_1)}$.

     Let $P$ be a pair-of-pants chosen from the pants decomposition $\pants$. Note that the ideal vertices, denoted $\mathcal{V}$, of ideal triangles associated to $P$ are all contained in $\partial \pi_1(\mathrm{S}_1)$ since $\pi_1(P) \subset \pi_1(\mathrm{S}_1)$. Thus, the images of vertices in $\mathcal{V}$ under $\xi_\rho$ and $\xi_{\Phi_t(\rho)}$ are the same. Since the triangle-shear invariants of $P$ are determined by the image of $\mathcal{V}$ (see \cite[Section 2]{Bonahon2014}), the internal parameters are invariant pointwise under $\Phi_t$.
\end{proof}

\begin{lemma}\label{lem:key}
Let $\mathrm{S}$ be a closed oriented surface of genus $g>1$. Let $\alpha$ be an oriented non-separating essential simple closed curve on $\mathrm{S}$. Let $\epsilon>0$. Then there exists a nonempty relatively compact open set $N$ contained in the $\epsilon$-thick part $\Hit^\epsilon_n(\mathrm{S})$ satisfying the following properties provided that $n>2$.   
\begin{enumerate}
    \item\label{inthick} $\Phi_t(N)\subset \Hit^\epsilon _n(\mathrm{S})$ for all $t\in \bR$.
    \item\label{nondegenerate} For any $t,s \in \bR$, $\left[\rho_1\right],\left[\rho_2\right] \in N$ and any $i=1,2, \cdots , 3g-3$, there exists a positive real number $\zeta>0$ which satisfies 
    \begin{align*}
        |\ell_{\Phi_t(\left[\rho_1\right])}(\alpha_i)-\ell_{\Phi_s{(\left[\rho_2\right])}}(\gamma)|>\zeta
    \end{align*}
    for every essential simple closed curve $\gamma$ that is not isotopic to $\alpha_i$.
    \item\label{fundomain1} $\{\phi \in \Mod \left(\mathrm{S}\right) \mid \phi_{*}\left(N\right) \cap N \neq \emptyset \} = \{ 1 \}$. 
    \item\label{fundomain2} Let $\mathrm{S}_1=\mathrm{S} \setminus \nu(|\alpha|)$.
    \[
    \{ \phi \in \Mod(\mathrm{S}_1, \partial \mathrm{S}_1) \mid \phi_*(\mathcal{R}(N)) \cap \mathcal{R}(N) \neq \emptyset \} = \{ 1 \}
    \]
    where $\mathcal{R}$ is the restriction map defined in Section~\ref{restriction}.
\end{enumerate}
\end{lemma}

\begin{proof}
    Parametrize $\Hit_n(\mathrm{S})$ with a real analytic diffeomorphism 
   \[\Xi:\Hit_n(\mathrm{S}) \to (\mathfrak{a}^+)^{(3g-3)} \times \mathcal{I}\times  \mathcal{G}
    \]
    as in Section~\ref{parametrization}. 
    
    Let 
    \[
    M_j=\diag\left((n-1)j\epsilon,(n-2)j\epsilon,\cdots,2j\epsilon,j\epsilon,-\frac{n(n-1)}{2}j\epsilon\right)
    \]
    for $1\leq j \leq 3g-3$, and let 
    \[
        L_j := \left\{X \in \mathfrak{a}^+  \mid |\omega_k(M_j) - \omega_k(X)|< \frac{\epsilon}{3}  \ \text{for all} \ 1\leq k \leq n \right\}.
    \]
    Since $M_j\in\mathfrak{a}^+$, we see that $L_j \neq \emptyset$.
    
    For every $j$ and every $M\in L_j$ with $n>2$, we have
\begin{equation}\label{eq:lowerbound1}
\omega_1(M)-\omega_n(M)>\frac{(n-1)(n+2)}{2}j\epsilon - \frac{2}{3}\epsilon>\epsilon.
\end{equation}
Also, for $M\in L_j$, we have 
\[
\omega_{k}(M)-\omega_{k+1}(M)>\left(j-\frac{2}{3}\right)\epsilon>0
\]
for $1 \leq k \leq n-2$, and 
\[
\omega_{n-1}(M)-\omega_n(M)>\left(\frac{n(n-1)}{2}+1\right)j\epsilon - \frac{2}{3}\epsilon>0.
\]

Choose a compact exhaustion $\{B_{j,i}\}_{i=1,2\cdots}$ in $\mathcal{I}_j$, that is, a nested sequence of compact sets satisfying $B_{j,i}\subset B_{j,i+1}^{\circ}$ and $\bigcup_i B_{j,i}=\mathcal{I}_j$ for every $j=1,\ldots,2g-2$.

Consider the inverse image 
    \[
    N_i = \Xi^{-1}\left(\prod^{3g-3}_{j=1}L_j
    \times \prod_{j=1}^{2g-2}(B_{j,i+1}^\circ\setminus B_{j,i})
    \times  \mathcal{G}\right).
    \]
    We show that, for some large $I$, $N_I$ satisfies (\ref{nondegenerate}), and $N_I \subset \Hit_n^{\epsilon}(\mathrm{S})$.

Let $K_i:=\inf \{ K([\rho]) \mid [\rho] \in N_i \}$, where $K(\rho)$ is the function that appeared in Theorem~\ref{thm:zhang1}. Since $\{ K([\rho]) \mid [\rho] \in N_i \} \geq 0 $, $K_i$ is well-defined. 

\begin{claim}
The set $\{K_i\}$ is unbounded. 
\end{claim}
\begin{proof}
Suppose that $\{K_i\}$ is bounded. Passing to a subsequence if necessary, we may assume that $K_i$ converges. Then, for each open set $N_i$, we can choose $[\rho_i] \in N_i$ such that $K_i \leq K([\rho_i]) < K_i + \frac{1}{i}$. Observe that, by construction, the sequence $\{[\rho_i]\}_{i=1}^{\infty}$ is an internal sequence.  Thus, the convergence of $K([\rho_i]) - K_i \to 0$ contradicts Theorem~\ref{thm:zhang2}. This implies that $\{K_i\}$ is unbounded.
\end{proof} 

Since $\{K_i\}$ is unbounded, we can choose a subsequence $\{K_{i_k}\}$ such that $K_{i_k}\to \infty$ as $k\to \infty$.  Let $\{N_{i_k}\}_{k=1}^\infty$ be the corresponding subsequence. 

Set 
\[
\zeta =\frac{(n-1)(n+2)}{2}\epsilon - 2\epsilon >0.
\]
By Theorem~\ref{thm:zhang1} and the claim above, we can find a large enough $k_0$ such that
\begin{equation}\label{eq:lowerbound2}
\length_{\rho}(\gamma) > \frac{(n-1)(n+2)(3g-3)}{2}\epsilon + \epsilon+\zeta
\end{equation}
for all $[\rho]\in N_{i_{k_0}}$ and all essential closed curves $\gamma$ not isotopic to an element of $\pants$. Let $I=i_{k_0}$. 

\emph{Case 1.} Assume that $\gamma$ is non-isotopic to any elements of $\mathcal{P}$. Then, for any $\left[\rho_1\right],\left[\rho_2\right] \in N_I$, 
    \[
    |\length_{\rho_1}(\alpha_i)-\length_{\rho_2}(\gamma)|
    = \length_{\rho_2}(\gamma)-\length_{\rho_1}(\alpha_i) 
    > \zeta
    \]
since $\length_{\rho_1}(\alpha_i) < \frac{(n-1)(n+2)(3g-3)}{2}\epsilon + \epsilon$ for all $1 \leq i \leq 3g-3$.

\emph{Case 2.} Assume that $\gamma$ is isotopic to some $\alpha_{j} \in \mathcal{P}$ with $j\ne i$. Then for any $\left[\rho_1\right],\left[\rho_2\right] \in N_I$,
\[
    |\length_{\rho_1}(\alpha_i)-\length_{\rho_2}(\alpha_j)| > \frac{(n-1)(n+2)|i-j|}{2}\epsilon - 2\epsilon > \frac{(n-1)(n+2)}{2}\epsilon - 2\epsilon = \zeta.
\]

By (\ref{eq:lowerbound1}), (\ref{eq:lowerbound2}), and the construction of $N_I$, we know that
\begin{equation}\label{eq:inthick}
N_I \subset \Hit_{n}^{\epsilon}(\mathrm{S}).
\end{equation}

By definition of the bulging flow and by Lemma~\ref{prop:main}, $(\mathfrak{a}^+)^{(3g-3)}$ and $\mathcal{I}$ are invariant pointwise under the bulging flow along a curve in $\mathcal{P}$. It follows that $\Phi_t(\left[\rho\right]) \in N_I$ for all $\left[\rho\right] \in N_I$ and all $t\in \bR$. Therefore $N_I$ satisfies (\ref{inthick}) from (\ref{eq:inthick}). 

Meanwhile, for any $t, s \in \bR$ and $\left[\rho_1\right],\left[\rho_2\right] \in N_I$, we obtain
\[
|\ell_{\Phi_t([\rho_1])} (\alpha_i) -\ell_{\Phi_s([\rho_2])} (\gamma)|>\zeta
\]
unless $\gamma$ is isotopic to $\alpha_i$ since $\Phi_t([\rho_1])$ and $\Phi_s([\rho_2])$ are still in $N_I$. Therefore, $N_I$ satisfies (\ref{nondegenerate}).

Thanks to Lemma~\ref{lem:2.3}, there exists a nonempty subset $N' \subset N_I$ satisfying (\ref{fundomain1}).

Choose any $[\rho] \in N'$. Since $\mathcal{R}(N') \cap \Delta$ is a nonempty open set, and since $\Mod(\mathrm{S}_1, \partial \mathrm{S}_1)$ acts properly and effectively on $\Delta$, we can apply  Lemma~\ref{lem:2.3}, to find an open neighborhood 
    $U\subset \Delta\cap \mathcal{R}(N')$ such that 
    \[
        \{\phi \in \Mod(\mathrm{S}_1, \partial \mathrm{S}_1) \mid \phi_*(U) \cap U \neq \emptyset\}=\{1\}.
    \]
    We choose any relatively compact, open subset $N\subset \mathcal{R}^{-1}(U)\cap N'$. Then $N$ is the desired open set satisfying (\ref{inthick}) - (\ref{fundomain2}).
    \end{proof}

Theorem~\ref{thm:main} can be shown by using the argument from \cite{Choi2025volume}. For convenience, we paraphrase their proof. In what follows, 
\[
\Pi:\Hit_n(\mathrm{S}) \to \mathcal{M}_n(\mathrm{S})=\Hit_n(\mathrm{S})/\Mod(\mathrm{S})
\]
denotes the projection map. 

\begin{lemma}[\cite{Choi2025volume}]\label{lem:subkey}
    Let $N$ be chosen as in Lemma~{\ref{lem:key}}.  Then there exists $D_N > 0$ with the following properties provided that $n>2$.
    \begin{enumerate}
        \item\label{notintersect1} $\Pi(N)\cap\Pi(\Phi_t(N)) = \emptyset$ provided $t > D_N$.
        \item\label{injective} $\Pi|\Phi_t(N)$ is injective for each t.
        \item\label{notintersect2} $\Pi(\Phi_t(N)) \cap \Pi(\Phi_s(N)) = \emptyset$ provided $|s-t| > D_N$.
    \end{enumerate}
\end{lemma}

\begin{proof}
The proof is almost identical to the proof of \cite[Lemma 4.9]{Choi2025volume}. 
      
    We first claim that, if there exists a homeomorphism $\phi \in \Mod(\mathrm{S})$ such that $\phi(N) \cap N \neq \emptyset$, then $\phi$ is a power of the Dehn twist along $\alpha$. To see this, observe that, by Lemma~\ref{lem:key}(\ref{nondegenerate}), $\phi$ preserves the isotopy class of the curve $\alpha$.  This enable us to view $\phi$ as an element of $\Mod(\mathrm{S}_1,\partial\mathrm{S}_1)$ where $\mathrm{S}_1= \mathrm{S}\setminus\nu(\alpha)$. Then by Lemma~\ref{lem:key}(\ref{fundomain2}), $\phi$ represents the trivial element in $\Mod(\mathrm{S}_1,\partial \mathrm{S}_1)$. Therefore, the claim follows. 
    
    Meanwhile, by \cite[Proposition~3.2]{Choi2025volume}, we know that there exists $M>0$ such that $\Phi^{Q,\alpha}(N) \cap N = \emptyset$ for all $Q\in \mathfrak{a}$ with $\|Q\| > M$. 
    
    Now if (\ref{notintersect1}) fails, then for any $n\in \mathbb{N}$, there exist $t_n>n$, a power of the Dehn twist $\phi_n$ along $\alpha$ and $[\rho_n]\in N$ such that $\phi_n\cdot \Phi_{t_n}([\rho_n])\in N$. Note that for a non-separating essential simple closed curve $\alpha$, the Dehn twist along $\alpha$ acts as  $[\rho]\mapsto \Phi ^{\lambda(\rho(\alpha)), \alpha}([\rho])$, where
      \[
      \lambda(\rho(\alpha))=\diag(-\lambda_1(\rho(\alpha)), \cdots, -\lambda_n(\rho(\alpha))).
      \] 
      It follows that there exist integers $m_n\in \mathbb{Z}$ such that  
      \[
      \Phi^{t_nQ+m_n\lambda(\rho(\alpha)),\alpha}([\rho])\in N,
      \]
      where  $Q=\diag(1,-2,1)$ or $Q=\diag(1,0,0,\cdots,0,-1)$. Note that $Q$ lies on the boundary of $\overline{\mathfrak{a}^+}$ while 
      \[
      \{\lambda(\rho(\alpha))\mid [\rho]\in \overline{N}\}\subset  -\mathfrak{a}^+.
      \]
      Since the closure $\overline{N}$ is compact, we observe that 
      \[
      \|t_nQ + m_n\lambda(\rho(\alpha))\| \to \infty
      \]
      as $t_n \to \infty$. This implies that $\Phi^{t_nQ+m_n\lambda(\rho(\alpha)),\alpha}(N) \cap N  = \emptyset$ for all large $n$, a contradiction. Thus (\ref{notintersect1}) holds.

      (\ref{injective}) and (\ref{notintersect2}) follow from Lemma~\ref{lem:key}(\ref{fundomain1}) and the fact that the Dehn twist action along $\alpha$ and the bulging flow $\Phi_t$ along $\alpha$ commute \cite[Lemma~4.1]{Choi2025volume}. 
\end{proof}

\begin{proof}[Proof of Theorem~\ref{thm:main}]
    Let $N \subset \Hit_n^\epsilon(\mathrm{S})$ be chosen as in Lemma~\ref{lem:key} and let $D_N$ be the constant from Lemma~\ref{lem:subkey}. We choose a sequence $0<t_1<t_2<\cdots$ of reals so that $t_{i+1} - t_i > D_N$. Lemma~\ref{lem:key}(\ref{inthick})  and Lemma~\ref{lem:subkey} together shows that $\Pi(\Phi_{t_i}(N))$ are open and pairwise disjoint in $\mathcal{M}^\epsilon_n(\mathrm{S})$. Since $\Phi_t$ is volume preserving, $\Pi(\Phi_{t_1}(N)), \Pi(\Phi_{t_2}(N))$, are pairwise disjoint and have identical positive volume in $\mathcal{M}_n^\epsilon(\mathrm{S})$.
\end{proof}

\bibliographystyle{alpha}
\bibliography{ref}

\end{document}